\documentclass[12pt]{amsart} 
\usepackage{amsfonts}
\usepackage{amsmath}
\usepackage{amssymb}
\usepackage{amsthm}
\usepackage{graphicx}
\usepackage[T1]{fontenc}

\usepackage{hyperref}

\newtheorem{theorem}{Theorem}

\newtheorem{observation}[theorem]{Observation}
\newtheorem{corollary}[theorem]{Corollary}
\newtheorem{problem}[theorem]{Problem}
\newtheorem{claim}[theorem]{Claim}

\begin{document}
	\title{On Strong Majority Edge Colourings with Few Colours}
\thanks{Research partially supported by the AGH University of Krakow under grant no. 16.16.420.054 and Excellence Initiative – Research University AGH University of Krakow, funded by the Polish Ministry of Science and Higher Education.}
	
	\author{Pawe{\l} P\k{e}ka{\l}a}
	\address[agh]{AGH University of Krakow, al. A. Mickiewicza 30, 30-059 Krakow, Poland}
	\email{\tt ppekala@agh.edu.pl}
	
	\author{Jakub Przyby{\l}o}
	\thanks{Corresponding author: Jakub Przyby{\l}o.}
	\address[agh]{AGH University of Krakow, al. A. Mickiewicza 30, 30-059 Krakow, Poland}
	\email{\tt jakubprz@agh.edu.pl}

\begin{abstract}
A strong majority edge colouring of a graph $G$ is an edge colouring in which, for every edge $e$ and every colour $\alpha$, at most half the edges adjacent to $e$ receive colour $\alpha$. Like many related colouring notions, it admits a natural interpretation as a colouring problem for an associated hypergraph. Somewhat surprisingly, although the corresponding hypergraph may have arbitrarily large vertex degrees, a universal finite upper bound on the sufficient number of colours in a strong majority edge colouring exists under a natural modest minimum degree assumption, unlike in several other closely related majority concepts. 

We in particular prove that every graph $G$ with minimum degree $\delta\ge5$ admits a strong majority edge colouring with three colours, improving both the previously known bound $\delta\ge9$ for three colours and the result showing that four colours suffice whenever $\delta\ge5$. 
Our result is best possible with respect to the number of colours and the minimum degree assumption.

We also introduce a more general framework of strong $1/k$-majority edge colourings and establish corresponding bounds for this setting.
\end{abstract}
	
			\maketitle

	\section{Introduction}

Majority colourings belong to the broad family of improper colourings in which
local conflicts are permitted but their extent is restricted. Instead of requiring
certain adjacent objects (in particular, vertices or edges) to receive distinct
colours, one imposes a weaker condition: no colour may constitute a majority within
the local neighbourhood of any considered object. A number of variants of this
concept have attracted considerable attention in recent years.

For vertex colourings, a colouring is called a \emph{majority colouring} if every
vertex has at most half of its neighbours in its own colour. The roots of this
notion can be traced back to Lov\'asz's decomposition theorem~\cite{Lovasz},
which immediately implies that every finite graph admits a majority colouring
with only two colours. The corresponding concept for infinite graphs is closely
related to the unfriendly partition problem, for which several fundamental
existence results have been established over the years
~\cite{Aharoni,Berger,Bruhn,KalinowskiStawiski}.

More recently, the majority condition has been investigated in several other
natural settings. Kreutzer, Oum, Seymour, van der Zypen and Wood~\cite{digraph1}
introduced majority colourings of digraphs, where the above majority restriction
is imposed only on the out-neighbours of each vertex. They proved that four
colours always suffice and conjectured that three colours should be enough. See
also~\cite{Szabo-majority,digraph2,MajorityGeneralOur,digraph3,Knox-Samal} for
related results and generalisations.

Already in~\cite{digraph1}, the authors observed that no fixed finite number of
colours is sufficient under a stronger requirement that \emph{no} colour dominates the
entire out-neighbourhood of any vertex $v$, rather than just the colour assigned
to $v$ itself, even when restricted to digraphs with arbitrarily large minimum
out-degree (note such a colouring clearly cannot exist for digraphs containing
a vertex of out-degree one); see Lemma~5 in~\cite{digraph1}. An analogous
negative statement holds for simple graphs by essentially the same argument; see
e.g.~\cite{strong_kkpw}, where the corresponding notion of
\emph{strong majority vertex colouring} was formally introduced.

A variant of such stronger notion was earlier deployed by
Bock \emph{et al.}~\cite{Bock} in the setting of edge colourings. They defined a
\emph{majority edge colouring} as an edge colouring in which no vertex has a
majority of its incident edges receiving the same colour; note that such a
colouring cannot exist if a graph contains a vertex of degree $1$.
They proved, in particular, that every graph $G$ with minimum degree
$\delta\geq 2$ admits a majority edge $4$-colouring, and every graph with
$\delta\geq 4$ admits a majority edge $3$-colouring. This is best possible,
since neither the number of colours nor the minimum-degree assumption can be
reduced in a result of this form; see~\cite{Bock,majority_gen,majority_list}
for details and related generalisations.

Concepts discussed above can also be viewed as instances of a more general
hypergraph colouring problem. Suppose we wish to colour the vertices of a
hypergraph so that no edge contains a majority of vertices of the same colour.
Then the strong vertex colouring of a graph $G=(V,E)$ corresponds to colouring a
hypergraph $H$ whose vertex set is $V$ and whose edges are the neighbourhoods of
the vertices of $G$. Analogously, a majority edge colouring of a graph
$G=(V,E)$ corresponds to colouring a hypergraph $H'$ whose vertex set is $E$ and
whose each edge consists of all edges  incident with a single vertex in $G$.
As mentioned above, there is no constant number of colours sufficient for all
hypergraphs $H$ arising in this way, even if one restricts attention to graphs
$G$ with minimum degree bounded from below. On the other hand, a small constant
number of colours suffices for all hypergraphs $H'$, provided that we impose a
modest lower bound on the minimum degree of $G$, which corresponds to a lower
bound on the edge sizes in $H'$. This difference stems from the fact that every
vertex of $H'$, generated for the latter problem, has degree~$2$, whereas vertex degrees in 
auxiliary hypergraphs $H$, from the former problem,
may be arbitrarily large. In fact, a fairly standard probabilistic argument,
generalising the one from~\cite{Bock} for graphs, yields a more general
conclusion: a finite number of colours suffices for every hypergraph whose
maximum degree is bounded above by a fixed constant, provided that all edges of the hypergraph are
sufficiently large. This also follows from recent results on majority colourings
of hypergraphs obtained in~\cite{Majoity-hypergraphs}, after applying them to
the duals of hypergraphs one is interested in.

A very intriguing notion of strong majority edge colouring was recently
introduced by Kalinowski, Kamyczura, Pil\'sniak and Wo\'zniak~\cite{strong_kkpw}.
A \emph{strong majority edge colouring} of a graph $G$ is an edge colouring such
that for every edge $e$ and every colour $\alpha$, at most half of the edges adjacent
to $e$ receive colour~$\alpha$. The least number of colours in such a colouring is
denoted by $\operatorname{Maj}'(G)$.
One may easily observe that, in this setting, the natural auxiliary hypergraph
$H''$, defined analogously as $H, H'$ above, may have arbitrarily large vertex degrees.
Nevertheless, rather surprisingly, structural properties of this concept
admit establishing a constant upper bound on the required number of colours
for graphs without small-degree vertices.
More specifically, note first that a strong majority edge colouring exists if and only if $G$
does not contain a pendant path of length greater than one, as observed in~\cite{strong_kkpw}. 
Graphs satisfying this condition are called \emph{admissible}.
The authors of~\cite{strong_kkpw} proved that every admissible graph $G$
satisfies $\operatorname{Maj}'(G)\leq 8$, which in particular implies that
$8$ colours are sufficient for every graph with minimum degree
$\delta\geq 2$. They also improved this bound to $4$ colours when
$\delta\geq 7$ and to $3$ colours when $\delta\geq 9$. Recently,
Antoniuk, Prorok and Salia~\cite{prorok} strengthened the first two results by
showing that $5$ colours suffice for every admissible graph and that
$\operatorname{Maj}'(G)\leq 4$ whenever $\delta\geq 5$. In this paper, we
further improve and supplement these results by proving that $3$ colours are
already sufficient whenever $\delta\geq 5$; see
Theorem~\ref{MainTheoremPP} below. This matches the smallest possible number of
colours in such a theorem and  the best possible minimum-degree
assumption, since $\delta$ cannot be reduced below~$5$; see
Section~\ref{ConclRem} for details.

In the next section, we introduce the necessary notation and recall several
useful observations. The main result together with its proof is presented in
Section~\ref{MainSect}. In Section~\ref{GeneralisationSect}, we introduce
generalised strong majority edge colourings, following directions similar to those
investigated in
~\cite{digraph1,MajorityGeneralOur,digraph3,Knox-Samal,Bock}, and establish
the corresponding generalisation of
Theorem~\ref{strong_kkpw_main}. The final section contains several concluding
remarks and comments.

	\section{Preliminaries}
	
	To prove some results we shall use the notion of $\alpha$-majority edge colourings, introduced in~\cite{Bock} 
	and further investigated in \cite{majority_list}. 
	For every graph $G=(V,E)$, we say that  $\omega:E\to C$ is an \emph{$\alpha$-majority edge colouring} if $d_c(v)\leq \alpha d(v)$ for every $v\in V$ and $c\in C$, where $d_c(v)$ denotes the number of edges incident with $v$ which are coloured with $c$ by $\omega$. 
	The following result from~\cite{majority_list} shall be useful. 
	
	\begin{theorem}[\cite{majority_list}]\label{thm_galvin_alpha}
		For any integer $\ell\geq 2$ and $\alpha \in (0,1)$ such that $\alpha \ell > 1$, if a graph $G$ has a minimum degree $\delta \geq \frac{2\ell-2}{\alpha \ell-1}$, then $G$ has an $\alpha$-majority edge colouring from lists of size $\ell$.
	\end{theorem}
	
	To simplify the notation, we shall say that $\omega:E\to C$ is an \emph{$(1/2-\varepsilon)$-majority edge colouring} if $d_c(v)< d(v)/2$ for every $v\in V$ and $c\in C$. For any given graph $G$ with maximum degree $\Delta$ this is in particular equivalent to an $\alpha$-majority edge colouring for any $\alpha=1/2-\varepsilon$ with $\varepsilon \in \left(0; \frac{1}{2\Delta}\right]$.
	
	Note that a majority colouring is not necessarily also a strong majority colouring. For example, if $e=uv$ is coloured with $\alpha$, exactly $d(u)/2$ edges incident with $u$ and exactly $d(v)/2$ edges incident with $v$ are coloured with $\beta \neq \alpha$, then the edge $e$ is not strong majority coloured. However, if for every vertex $w$ and each colour $\alpha$, strictly less than $d(w)/2$ edges incident with $w$ are coloured with $\alpha$, 
	then for any edge $e=uv$ at most $(d(u)+d(v)-2)/2$ edges adjacent to $e$ may have the same colour. Hence, every $(1/2-\varepsilon)$-majority edge colouring is also a strong majority colouring. 
To exemplify utility of this observation as well as Theorem~\ref{thm_galvin_alpha}, we first briefly present how these imply the following result from~\cite{strong_kkpw}, mentioned above.
	
	\begin{theorem}[\cite{strong_kkpw}]\label{strong_kkpw_main}
		If $G$ is a graph with minimum degree $\delta \geq 9$, then
		\[ \operatorname{Maj}'(G) \leq 3. \]
		
		\begin{proof}
			Let $G$ be a graph with maximum degree $\Delta$ and minimum degree $\delta \geq 9$.
			Let $\alpha = \frac{1}{2} - \frac{1}{6\Delta}$. Then, $3\alpha > 1$ and
			\[ \frac{2\cdot 3-2}{\alpha \cdot 3-1} = \frac{4}{\frac{1}{2}-\frac{1}{2\Delta}} \leq \frac{4}{\frac{1}{2}-\frac{1}{18}} = 9 \leq \delta. \]
			Hence, by Theorem \ref{thm_galvin_alpha}, $G$ has an $\alpha$-majority edge colouring with $3$ colours, and thus $\operatorname{Maj}'(G)\leq 3$.
		\end{proof}
	\end{theorem}

These tools shall be used again in Section~\ref{GeneralisationSect}.
	To prove our main result we shall in turn in particular exploit the following observation, which is a direct consequence of Euler's Theorem, see e.g.~\cite{Bock} for details of the proof.
	
	\begin{observation}
		\label{euler}
		Let $G$ be a connected graph.
		\begin{enumerate}
			\item[$(1^\circ)$] If $G$ has an even number of edges or $G$ contains vertices of odd degree, then 
			$G$ has an edge $2$-colouring such that 
			for every vertex $u$ of $G$, at most $\lceil d(u)/2 \rceil$ of the edges incident with $u$ have the same colour.
			\item[$(2^\circ)$] If $G$ has an odd number of edges, all vertices of $G$ have even degrees and $u_G$ is any fixed vertex of $G$, then $G$ has an edge $2$-colouring such that 
			for every vertex $u$ of $G$ distinct from $u_G$, exactly $d_G(u)/2$ of the edges incident with $u$ have the same colour, and at most $1+d(u_G)/2$ of the edges incident with $u_G$ have the same colour.
		\end{enumerate}
	\end{observation}
	
	In further applications, a \textit{bad vertex} shall mean a vertex of $G$ which was chosen as the vertex $u_G$ while applying Observation~\ref{euler}, that is the vertex with exactly $1+d(v)/2$ incident edges coloured the same in one of the two colours.

			\section{Main result}\label{MainSect}
	
	\begin{theorem}\label{MainTheoremPP}
		If $G$ is a graph with minimum degree $\delta \geq 5$, then 
		\[ \operatorname{Maj}'(G) \leq 3. \]
		\begin{proof}
			Let $G=(V,E)$ be a graph with minimum degree $\delta\geq 5$. We shall try to construct a $(1/2-\varepsilon)$-majority edge colouring of $G$ using colours in $\{0,1,2\}$. 
			We shall need to admit an exception, though, at the end of our argument.
			In the first part of the proof, we perform several modifications of our graph in order to choose a special set of edges to be coloured with $0$ in $G$; a small number of additional edges shall possibly be also coloured with $0$ in the concluding part of the proof.
			
			We begin by performing a splitting operation: we split every vertex of $G$ into vertices with degrees at most $6$. More precisely, given any vertex $v$ of $G$, we write its degree as $d(v) = 6(p-1)+i$ where $i\in[6]$, i.e. $p=\lceil d(v)/6\rceil$. We then partition the neighbourhood of $v$ in $G$ into $p$ disjoint sets $N_1, \dotsc, N_p$ such that $|N_1|=i$ and $|N_j|=6$ for $j\geq 2$. Let $G'$ be a graph with vertex set $V' = (V\smallsetminus\{v\}) \cup \{v_1,\dotsc,v_p\}$  and edge set $E' = E(G-v)\cup \bigcup_{i=1}^{p} \{uv_i : u\in N_i\}$, where $v_1,\dotsc,v_p$ are new vertices, which can be thought of as copies of $v$ resulting from its splitting. Note that this operation yields a natural bijection between the edges $E'$ of $G'$ and the edges $E$ of $G$. In addition, 
			\begin{eqnarray}
 			&& \text{while splitting any
		vertex $v$ of degree $8$ in $G$,
			we ensure} \nonumber\\
			&& \text{that its copy $v_1$ of degree $2$ is adjacent to as many vertices} \nonumber\\
			&& \text{which are not of degree $6$ in $G$ as possible.} \label{SplittingCondition}
			\end{eqnarray}
			 In other words, if the corresponding set $N_1$ contains a vertex of degree $6$ in $G$, then all vertices in the associated $N_2$ must be of degree $6$ in $G$. 	

			Let $S$ be a graph constructed from $G$ by applying the above splitting operation to all vertices of $G$, one after another (in arbitrary order). If $S$ contains a vertex of odd degree, we add a new vertex $s$ and join it with an edge with every vertex of odd degree in $S$. Every connected component of the resulting graph $S'$ is Eulerian and therefore admits an Eulerian circuit.
For each component, fix one such circuit, traverse it once, and order the edges according to the
direction of traversal.
			Then remove the added vertex $s$ along with all its incident edges. In the resulting directed graph $D$, every vertex of even degree in $S$ has the same number of incoming and outgoing arcs, while for every vertex of odd degree in $S$ those numbers differ by exactly one. Note again that a natural bijection exists between the arc set of $D$ and the edge set of $S$ (hence also the edge set of $G$).
			
			Let $T$ be the set of vertices of degree $2$ in $S$ which are copies of vertices of degree $8$ in $G$. Note that every connected component of the subgraph induced by $T$ in $S$ is either a cycle or a path (possibly including just one vertex). If a component is a cycle, then remove all of its vertices from $D$. Otherwise, it is a path, which must have become an oriented path in $D$. More precisely, for any such path $t_1t_2\ldots t_r$ in $S$ there are vertices $u$, $w$ such that $\overrightarrow{ut_1},\overrightarrow{t_1t_2},\overrightarrow{t_2t_3},\ldots,\overrightarrow{t_{r-1}t_r},\overrightarrow{t_rw}$ 
			are arcs in $D$.
			For every such extended oriented path (or cycle) $ut_1t_2\ldots t_rw$  in $D$, we remove all its internal vertices $t_1,t_2,\ldots,t_r$. That is, we delete the vertices $t_1,t_2,\ldots,t_r$ and the arcs $\overrightarrow{ut_1},\overrightarrow{t_1t_2},\overrightarrow{t_2t_3},\ldots,\overrightarrow{t_{r-1}t_r},\overrightarrow{t_rw}$ 
and subsequently we add a new arc $\overrightarrow{uw}$. 
Note that this way multiple arcs may arise, as well as loops (for $u=w$). Let $D'=(U,A)$ be the resulting directed multigraph, which admits loops, each of which increases in and out-degrees of its vertex by one. By our construction, the vertices of $D'$ retain the same in and out-degrees as in $D$, and only vertices of degree $2$ in $S$ which are copies of vertices of degree $8$ in $G$, which were present in $D$, are missing in $D'$.
			
			Next, we construct a bipartite multigraph $B=(U'\cup U'', F)$ from $D'$  
			such that $U'$, $U''$ are copies of $U$, say $U'=\{u':u\in U\}$, $U''=\{u'':u\in U\}$, 
			and we include an edge $v'w''$ in $F$ if and only if $\overrightarrow{vw}\in A$. 
			 Note the resulting bipartite multigraph has maximum degree at most $3$, and has no loops. Hence, it admits a proper edge colouring $c$ with colours $\{3,4,5\}$ (see \cite{Konig}). These colours are temporary and serve only defining the auxiliary colouring $c$. Let us remove all edges coloured $3$ from $B$. The resulting bipartite multigraph $R$ is clearly a collection of paths and even cycles. (We may now discard the temporary colours from all the edges.)
We next restore (some of) the previously removed vertices of degree $2$ originating from vertices of degree $8$.
Namely, for every edge $u'w''$ in $R$ corresponding to an arc  $\overrightarrow{uw}$ 
of $D'$ which was obtained from some oriented path $ut_1t_2\ldots t_rw$ in $D$, we remove the edge $u'w''$ 
 and add vertices $t_1,t_2,\ldots, t_r$ along with the edges $u't_1,t_1t_2,t_2t_3,\ldots,t_{r-1}t_r,t_rw''$.
The resulting graph $R'$ is a collection of paths and cycles, possibly including odd length cycles.
			
			Subsequently, we choose a maximum matching $M$ in $R'$ such that every vertex of degree $2$ in $R'$ which did not originate from a vertex of degree $8$ in $G$ is incident with an edge of $M$. 
Such a matching can be easily chosen component by component.
Indeed, consider any connected component $R''$ of $R'$.
If $R''$ is an even length cycle or an odd length path, we may choose its every second edge to be included in $M$; consequently, all vertices of $R''$ are incident with $M$. We proceed the same way if $R''$ is an even length path, leaving a single vertex of degree $1$ in $R''$ without an incident edge in $M$. Finally, if $R''$ is an odd length cycle, it must contain at least one vertex $t$ of degree $2$ originating from a vertex of degree $8$ in $G$. We then choose a maximal matching in $R''$ whose edges are not incident with $t$. Clearly, the resulting $M$ meets our requirements.
			
		It is straightforward to track back that every edge $e$ of $M$ arose from a different edge $e'$ in our original graph $G=(V,E)$ (in other words, there is a naturally defined injection from $M$ to $E$).
		For simplicity, we shall abuse our notation slightly, 
		and denote also by $M$ the set of all the corresponding edges $e'$ in $G$, and similarly -- the set of their counterparts in  
		the split graph $S$ of $G$, etc.
		We colour all these edges with $0$.
		Consider any given vertex $v$ in $S$.
		Note it can be incident with at most two edges coloured $0$ in $S$ (one for every copy of $v$ in $B$).
		Moreover, 
		\begin{eqnarray}
		&&\text{if $v$ has degree $1$ in $S$ or has degree $2$ in $S$}\nonumber\\ 
		&&\text{and originates from a vertex of degree $8$ in $G$, then}\nonumber\\ 
		&&\text{at most one edge coloured $0$ is incident with $v$ in $S$.}\label{One-0-incidence}
		\end{eqnarray}
		 On the other hand, all vertices of degree $3$ in $B$, which necessarily had degree $2$ in $R'$, are guaranteed to have an incident edge coloured $0$. Consequently, every vertex of degree $5$ in $S$ has at least one incident edge coloured $0$, while each vertex of degree $6$ in $S$ has exactly two such edges.
			
			Let us observe that colour $0$ does not violate the $(1/2-\varepsilon)$-majority edge colouring condition in the original graph $G$.
			Indeed, consider any given vertex $v\in V$. 
			Write its degree in $G$ as $d(v) = 6l+i$ where $l$ is a nonnegative integer and $i\in \{ 0,1,2,3,4,5 \}$. (Note that since $\delta\geq 5$, we may have $l=0$ only if $i=5$.) It suffices to show that $v$ is incident with at most $3l+\lfloor\frac{i-1}{2}\rfloor$ edges coloured $0$. 
			If $i=0$, then exactly $2l \leq 3l-1$ edges incident with $v$ are coloured $0$. Analogously, if $i=1$, then at most $2l+1 \leq 3l$ edges incident with $v$ are coloured $0$. 
			In the remaining cases there are at most $2l+2$ such edges.  
			Thus, for $i \in \{3,4\}$, it suffices to note that $2l+2 \leq 3l+1$, while for $i=5$ -- that $2l+2 \leq 3l+2$.
			 Further, if $i=2$ and $l\geq 2$, then $2l+2 \leq 3l$. The only remaining case is $d(v)=8$, but by our construction, in particular~(\ref{One-0-incidence}), such $v$ can be incident with at most $3=3l$ edges coloured $0$. Therefore, in all cases the $(1/2-\varepsilon)$-majority edge colouring condition holds for colour $0$. 
			
			Let $H$ be the subgraph obtained from $G$ by removing all edges of $M$, that is all edges already coloured with $0$. 
			In what follows, we aim to 
			apply Observation~\ref{euler} to colour the edges of $H$ with colours $1$ and $2$. 
			This may require colouring some specific additional edges with $0$.
			Note that for any vertex $v$ of $H$, if $d_G(v) = 6l+i$ with $i\in \{ 0,1,2,3,4,5 \}$, then $d_H(v) \leq 4l+i$ if $i\neq 5$ and $d_H(v) \leq 4l+4$ otherwise. Hence, by directly applying Observation \ref{euler}, at most $2l + \lfloor\frac{i}{2}\rfloor + 1$ edges incident to $v$ would get the same colour in case $v$ is a bad vertex. Otherwise, at most $2l + \lceil\frac{i}{2}\rceil$ edges would receive the same colour if $i \neq 5$, and at most $2l + 2$ if $i=5$. Comparing these bounds with the number $3l+\lfloor\frac{i-1}{2}\rfloor$ of uniformly coloured edges 
			allowed by the $(1/2-\varepsilon)$-majority edge colouring condition, we conclude the following.
			
			If we apply Observation~\ref{euler} to a component of $H$ containing any given vertex $v$, then the $(1/2-\varepsilon)$-majority edge colouring condition holds for $v$ (with respect to all colours, $0$, $1$, $2$),  
				unless $v$ is a bad vertex and one of the following holds:
				\begin{itemize} 
				\item[(i)] $d_H(v)=4$ and $d_G(v)=5$;
				\item[(ii)] $d_H(v)=4$ and $d_G(v)=6$;
				\item[(iii)] $d_H(v)=6$ and $d_G(v)=8$;
				\item[(iv)] $d_H(v)=8$ and $d_G(v)=10$.
			\end{itemize}
			For every connected component $H'$ of $H$ we then proceed as follows. 
			If $H'$ has an even number of edges or contains a vertex of odd degree, then we apply Observation \ref{euler} to $H'$ to colour its edges with colours $1$ and $2$. Otherwise, if $H'$ contains a vertex $v$ for which none of the conditions (i) -- (iv) above holds, then 
			we analogously apply Observation~\ref{euler} to $H'$ choosing $v$ to be a bad vertex.
			By the observations above, the resulting colouring satisfies the $(1/2-\varepsilon)$-majority edge colouring condition for all vertices of $H'$.
			
			We may thus assume that $H'$ has an odd number of edges and conditions (i) -- (iv) hold for all vertices of $H'$. Consequently, $H'$ must contain a vertex $u$ with $d_H(u)=6$ and $d_G(u)=8$.
			
			Suppose first that $u$ has a neighbour $v$ in $H'$ for which (ii) does not hold, i.e. it is one of the types (i), (iii) or (iv). Then we colour $uv$ with $0$ and remove it from $H$. Note that consequently, the $(1/2-\varepsilon)$-majority edge colouring condition still holds for $u$, $v$ and colour $0$. Moreover, the modified $H'$ gains two vertices of odd degree and remains connected (since no graph can have a single vertex of odd degree). Thus, we may apply Observation~\ref{euler} to the resulting component and obtain a 
			colouring with $1$ and $2$ so that the $(1/2-\varepsilon)$-majority edge colouring condition is satisfied for all its vertices and each of the colours $0$, $1$ and $2$.
			
			We may thus finally assume that (ii) holds for all neighbours $v$ of $u$ in $H'$.
			Note that originally, $u$ was split into vertices of degree $6$ and $2$ in $S$. Since $d_H(u)=6$, this means that neither of the edges incident with the copy $u_1$ of $u$ with $d_S(u_1)=2$   
			was coloured with $0$, and hence they join $u$ with vertices of degree $6$ in $G$. 
			However, by our special splitting condition~(\ref{SplittingCondition}), this implies that 
			all the neighbours of $u$ in $G$ have degree $6$ in $G$. In this case, we can no longer guarantee that our colouring is a $(\frac{1}{2}-\varepsilon)$-majority edge colouring for $u$. Instead, we 
			apply Observation~\ref{euler} to $H'$ choosing $u$ to be a bad vertex.
			As a result, $u$ ends up with $4$ incident edges coloured $1$ (or $2$).
			However, by our construction (as an exemption from the $(1/2-\varepsilon)$-majority edge colouring condition shall be admitted only for some vertices of degree $8$ in $G$), all its neighbours shall end up with exactly two incident edges in each of the colours.
			Since any edge between vertices of degree $6$ and $8$ allows $\frac{6+8-2}{2}=6$ edges with the same colour, the resulting colouring shall be a strong majority edge colouring.
		\end{proof}
	\end{theorem}
	
	Note that the last case in the above proof was the only one which prevented obtaining a $(1/2-\varepsilon)$-majority edge colouring of $G$. As it concerned degree $8$ vertices with all neighbours of degree $6$, the same proof implies the following.
	\begin{corollary}
		Let $G$ be a graph with maximum degree $\Delta$ and minimum degree $\delta \geq 7$. Then, for any $\varepsilon \in \left(0; \frac{1}{2\Delta}\right]$, $G$ has a $(1/2-\varepsilon)$-majority edge colouring with $3$ colours.
	\end{corollary}
		Next observation shows that this result is best possible.
	\begin{observation}
		There exists a graph $G$ with minimum degree $\delta = 6$ which is not $(1/2-\varepsilon)$-majority edge colourable using $3$ colours.
		\begin{proof}
			We construct a graph $G$ as follows. First, take a complete graph on $8$ vertices and remove the edges of any fixed Hamilton cycle from it. Then, add an additional vertex $v$ and connect it with all remaining vertices. Note that $d(v)=8$ and all other vertices are of degree $6$.
			
			Suppose there is a $(1/2-\varepsilon)$-majority edge $3$-colouring of $G$. Note that in such a colouring, a vertex of degree $6$ must have exactly two incident edges in each colour. On the other hand, a vertex of degree $8$ allows at most $3$ incident edges coloured uniformly. Thus, there must exist a colour $\alpha$ such that exactly $3$ edges incident with $v$ are coloured with it. Consider the subgraph of $G$ induced by the edges of $G$ coloured with $\alpha$. This subgraph has eight vertices of degree $2$ and a single vertex of degree $3$, which is not possible.
		\end{proof}
	\end{observation}

	\section{Generalised strong majority edge colourings}\label{GeneralisationSect}
	
	Following multiple natural extensions of the majority colouring concepts, studied e.g. in~\cite{digraph1,MajorityGeneralOur,digraph3,Knox-Samal,Bock},
	we propose the following notion. The colouring $c$ is called a \textit{strong $\frac{1}{k}$-majority edge colouring} if for every colour $\alpha$ and every edge $e$, at most a $1/k$ fraction of the edges adjacent to $e$ have the same colour $\alpha$. If $e=uv$ this implies that at most $(d(u)+d(v)-2)/k$ 
	 edges adjacent to $e$ have the colour $\alpha$. Note such a colouring of a non-empty graph $G$ may exist only if $(d(u)+d(v)-2)/k\geq 1$  for every edge $uv$, which is equivalent to $d(u)+d(v) \geq k+2$.
	 On the other hand, 
	if such a condition holds for every edge $uv$ of $G$, then a strong $\frac{1}{k}$-majority edge colouring always exists, since, in the worst case, one may assign a distinct colour to every edge.
	This in particular implies that a strong $\frac{1}{k}$-majority edge colouring exists for every graph with minimum degree $\delta\geq 1+k/2$.
	
	Using Theorem \ref{thm_galvin_alpha}, we provide below in Theorem~\ref{GenerelThPP} a lower bound for $\delta$ guaranteeing a strong $\frac{1}{k}$-majority edge colouring with the least number of $k+1$ colours; see the next section for further explanations and comments.
	Note this is a direct generalisation of Theorem~\ref{strong_kkpw_main} from~\cite{strong_kkpw}, which covers the case of $k=2$ exclusively.
	To prove it we  shall exploit a natural extension of the previously utilized fact, namely that an $\alpha$-majority edge colouring of a graph 
	$G$ with $\alpha = \frac{1}{k}-\varepsilon$ and $\varepsilon \in \left(0; \frac{1}{k\Delta}\right]$, where $\Delta$ denotes the maximum degree of $G$, is also a strong $\frac{1}{k}$-majority edge colouring of $G$.	
	
	\begin{theorem}\label{GenerelThPP}
		For every integer $k\geq 2$, each graph $G$ with minimum degree $\delta \geq 2k^2+1$ has a strong $\frac{1}{k}$-majority edge colouring with $k+1$ colours.
		\begin{proof}
			Let $G$ be a graph with maximum degree $\Delta$ and minimum degree $\delta \geq 2k^2+1$. 
			Set $\alpha = \frac{1}{k} - \frac{1}{k(k+1)\Delta}$. Then, $\alpha \cdot (k+1) > 1$ and
			\begin{align*}
				\frac{2 \cdot (k+1)-2}{\alpha \cdot (k+1)-1} &= \frac{2(k+1)-2}{(\frac{1}{k} - \frac{1}{k(k+1)\Delta})(k+1)-1} = \frac{2k}{\frac{1}{k}-\frac{1}{k\Delta}} = \frac{2k^2}{1-\frac{1}{\Delta}}\\
				&\leq \frac{2k^2}{1-\frac{1}{2k^2+1}} = 2k^2+1  \leq \delta.
			\end{align*}
			Hence, by Theorem \ref{thm_galvin_alpha}, $G$ has an $\alpha$-majority edge colouring with $k+1$ colours, and thus it has a strong $\frac{1}{k}$-majority edge colouring with $k+1$ colours.
		\end{proof}
	\end{theorem}

\section{Concluding remarks}\label{ConclRem}

Note that a graph $G$ may have a strong majority edge $2$-colouring only if each component of $G$ includes only vertices of even degree or only vertices of odd degree (and sometimes this is still not enough, as exemplified e.g. by cycles $C_n$ with $n$ not divisible by $4$, which require $3$ colours, see~\cite{strong_kkpw}). Thus, our Theorem~\ref{MainTheoremPP} is optimal with respect to the number of colours. 

More generally,
 a graph $G$ may have a strong $\frac{1}{k}$-majority edge colouring with $k$ colours if $d(u)+d(v)-2 \equiv 0 \mod k$ for every edge $uv$ of $G$. Hence, we cannot obtain a counterpart of Theorem~\ref{GenerelThPP} with a smaller number than $k+1$ colours either. 

The following simple observation shows that the lower bound for $\delta$ in Theorem~\ref{GenerelThPP} is of the right order of magnitude.
	\begin{observation}
		For every $k\geq 2$ there exists a graph $G$ with minimum degree $\delta = \frac{1}{2}(k^2-k)$ which is not strong $\frac{1}{k}$-majority edge $(k+1)$-colourable.
		\begin{proof}
			Let $G$ be any graph with minimum degree $\delta = (k^2-k)/2$ which contains an edge $e=uv$ such that $d(u) = \delta$ and $d(v)=\delta+1$. 
			Suppose there is a strong $\frac{1}{k}$-majority edge colouring of $G$ with $k+1$ colours.
			Then, at most $\lfloor \frac{\delta+(\delta+1)-2}{k} \rfloor = \lfloor \frac{k^2-k-1}{k} \rfloor = k-2$ edges adjacent with $e$ can be uniformly coloured. 
			Thus, at most  $(k+1)(k-2) = k^2-k-2$ edges adjacent to $e$ may be coloured at all, which is less than the total number of edges adjacent to $e$; a contradiction.
		\end{proof}
	\end{observation}
	
For $k=2$ we provide a better example, implying near optimal estimate for $\delta$.
	\begin{observation}\label{ObsK34}
	The complete bipartite graph $K_{3,4}$ is not strong majority edge $3$-colourable.
		\begin{proof}
			Let $G=(V,E)$ be the graph $K_{3,4}$. Suppose there exists its strong majority edge colouring with colours $1,2,3$. Since every edge of $G$ is adjacent with exactly $5$ edges, 
			\begin{eqnarray}
			&&\text{at most $2$ edges adjacent to any given edge}\nonumber\\
			&&\text{can be coloured the same.} \label{2edgesCond}
			\end{eqnarray}
			Let $G_i=(V_i,E_i)$ be the subgraph of $G$ induced by the edges coloured with $i$, for $i=1,2,3$.
			Consider any given $G_i$. We first observe that if $G_i$ has at least $4$ edges, then it must be isomorphic to the cycle $C_4$. 
			
			Let $\Delta$ denote the maximum degree of $G_i$. If $\Delta \geq 3$ and $G_i$ is not isomorphic to the star $K_{1,3}$ then we easily get a contradiction with~(\ref{2edgesCond}). If $\Delta = 1$ then $E_i$ is a matching, and hence has at most $3$ elements. It remains to consider the case of $\Delta = 2$. Let $u$ be a vertex of degree $2$ in $G_i$. If there is a vertex $v$ with degree at least $1$ in $G_i$ and $uv\in E$, then $uv \in E_i$, otherwise we would get a contradiction with~(\ref{2edgesCond}) for the edge $uv$. Thus, if $G_i$ has at least $4$ edges, then it has exactly $4$ edges and indeed is isomorphic to $C_4$.
			
Consequently, since $G$ has $12$ edges, it must have exactly $4$ edges in each of the colours. Hence, each colour induces $C_4$ in $G$, but this implies that all vertex degrees in $G$ are divisible by $2$; a contradiction.
		\end{proof}
	\end{observation}

Observation~\ref{ObsK34} implies that the lower bound for $\delta$ in Theorem~\ref{MainTheoremPP} 
cannot be decreased by more than one. We in fact suspected that $4$ is an achievable lower bound for $\delta$, and posed the corresponding conjecture in the first version of this paper. However, McNeil~\cite{McNeil} quickly found the following example (and various other examples) by computer search and informed us of the result.
\begin{claim}[\cite{McNeil}]\label{ObsK45}
	The complete bipartite graph $K_{4,5}$ is not strong majority edge $3$-colourable.
\end{claim}
We leave the proof of this fact to the reader -- one might, for example, follow the approach within the proof of Observation~\ref{ObsK34} to show that at most $6$ edges can be coloured uniformly in $K_{4,5}$.
Consequently, our main result, stated in Theorem~\ref{MainTheoremPP}, cannot be improved, even in a weaker version -- restricted to bipartite graphs exclusively.

It is tempting to try to generalise Theorem~\ref{MainTheoremPP} directly to larger values of $k$, and to ask whether every graph with minimum degree $\delta\geq k^2+1$ admits a strong $\frac{1}{k}$-majority edge colouring with $k+1$ colours.
However, we do not feel to have gained a substantial enough insight into the general variant of strong majority edge colourings to dare to pose such a conjecture. Rather than that, we leave this issue as a more generally open problem.
\begin{problem}
Given any fixed integer $k\geq 3$, what is the least $\delta_k$ such that
each graph $G$ with minimum degree $\delta \geq \delta_k$ has a strong $\frac{1}{k}$-majority edge colouring with $k+1$ colours?
\end{problem}

\section*{Acknowledgements}
The authors are grateful to D.~S. McNeil for finding 
the example establishing the sharpness of our result and for bringing it to our attention so promptly.

\end{document}